\documentclass[11pt]{article}

\usepackage{amsmath, amssymb, amsthm}
\usepackage[utf8]{inputenc}
\usepackage[T1]{fontenc}
\usepackage{geometry}
\usepackage{hyperref}
\usepackage{mathtools}

\newtheorem{theorem}{Theorem}[section]
\newtheorem{lemma}[theorem]{Lemma}
\newtheorem{proposition}[theorem]{Proposition}
\newtheorem{corollary}[theorem]{Corollary}
\theoremstyle{definition}
\newtheorem{definition}[theorem]{Definition}
\theoremstyle{remark}
\newtheorem{remark}[theorem]{Remark}

\newcommand{\C}{\mathbb{C}}
\newcommand{\R}{\mathbb{R}}
\newcommand{\D}{\mathbb{D}}
\newcommand{\Acal}{\mathcal{A}}
\newcommand{\Bcal}{\mathcal{B}}
\newcommand{\Mcal}{\mathcal{M}}
\newcommand{\Fcal}{\mathcal{F}}
\newcommand{\fa}{\mathfrak{a}}
\newcommand{\fb}{\mathfrak{b}}
\newcommand{\ff}{\mathfrak{f}}
\newcommand{\Wcal}{\mathcal{W}}

\title{The Framed Beltrami-Vekua Normal Form \\ and its Pseudo-Analytic Mass}
\author{Daniel Alay\'on-Solarz\thanks{\texttt{danieldaniel@gmail.com}}}
\date{June 2026}

\begin{document}

\maketitle

\begin{abstract}
We normalize a first-order real planar elliptic system, by pointwise algebra, to a \emph{framed Beltrami--Vekua equation} $\Phi(w_{\bar z} - \mu w_z) + \Psi(\overline{w_z} - \mu\,\overline{w_{\bar z}}) + \fa w + \fb \bar w = \ff$, with $|\mu| < 1$ and $|\Phi| > |\Psi|$, and compute the closed transformation laws of its data under the recombination of unknowns $w \mapsto \varphi w + \psi \bar w$ and under orientation-preserving $C^1$ changes of variables. The 2-form
\[
\Theta = \frac{\bigl|\,\Phi\,\fb - \Psi\,\fa - (\Phi\, L\Psi - \Psi\, L\Phi)\,\bigr|^2}{\bigl(|\Phi|^2 - |\Psi|^2\bigr)^2\,\bigl(1 - |\mu|^2\bigr)}\; dx\, dy, \qquad L = \bar\partial - \mu\,\partial,
\]
is invariant under the recombination and covariant under the changes of variables. The total mass $\Mcal = \int_\Omega \Theta$ is therefore an invariant of the equivalence class. One recombination and one scaling carry any framed equation, in closed
form, onto the trivial-frame slice --- a Beltrami--Vekua equation over the
same $\mu$ --- there identifying $\Theta$ with the pseudo-analytic mass
density of \cite{mass}. We then show all of this persists at measurable regularity: it suffices that $\mu$ be measurable and locally elliptic and that the frame lie in $W^{1,2}_{\mathrm{loc}} \cap L^\infty_{\mathrm{loc}}$, the changes of variables then being quasiconformal homeomorphisms. In that class every equation with $\|\mu\|_\infty < 1$ is quasiconformally equivalent, of equal mass, to one over $\mu = 0$.
\end{abstract}

\section{Introduction}\label{sec:intro}

The Beltrami--Vekua equation
\[
w_{\bar z} - \mu\, w_z + \Acal\, w + \Bcal\,\bar w = \Fcal, \qquad |\mu| < 1,
\]
of \cite{mass} is the universal complex form of a smooth first-order real planar
elliptic system that keeps the principal part inside the equation ---
\emph{pre-uniformized}, in the language of \cite{mass} --- while wearing
Vekua's algebraic dress: one principal-part term and $\bar w$, not
$\overline{w_{\bar z}}$, in the conjugate slot. On the space of such equations
\cite{mass} acts by two symmetries, a multiplicative gauge $w \mapsto \phi w$
of the unknown and an orientation-preserving change of variables of the
domain, and exhibits a 2-form, the \emph{pseudo-analytic mass density}
$|\Bcal|^2\,(1 - |\mu|^2)^{-1}\,dx\,dy$, invariant under the gauge and covariant
under the change of variables; its integral over $\Omega$ --- the
\emph{pseudo-analytic mass} --- is therefore an invariant of the equivalence class.

\medskip\noindent\textbf{The full recombination of the unknown.}
The gauge $w \mapsto \phi w$ is only half of the natural symmetry. A complex
unknown over a real system carries a real two-dimensional fibre, and the
largest pointwise group acting on it is the real-linear recombination
\[
w \;=\; \varphi\, w' + \psi\,\bar w', \qquad |\varphi| > |\psi|,
\]
of which $w \mapsto \phi w'$ is the holomorphic slice $\psi = 0$. The
Beltrami--Vekua class is \emph{not} closed under the missing half: feeding
$\bar w'$ through the operator $L = \bar\partial - \mu\,\partial$ produces a
term in $\overline{M w'}$, where $M = \partial - \bar\mu\,\bar\partial$, that
the single principal slot cannot hold. Restoring it costs one coefficient.
The resulting \emph{framed Beltrami--Vekua equation}
\[
\Phi\, L w \;+\; \Psi\, \overline{M w} \;+\; \fa\, w + \fb\, \bar w \;=\; \ff,
\qquad |\Phi| > |\Psi|,
\]
carrying a \emph{frame} $(\Phi, \Psi)$ in place of the normalized leading pair
$(1, 0)$, is the smallest class stable under the full recombination --- its
\emph{terminal} class. The Beltrami--Vekua equation of \cite{mass} is the
trivial-frame slice $(\Phi, \Psi) = (1, 0)$, and the recombination acts on the
frame by composition with the pointwise real-linear map
$z \mapsto \varphi z + \psi\bar z$, namely
$(\Phi, \Psi) \mapsto (\Phi\varphi + \Psi\bar\psi,\; \Phi\psi + \Psi\bar\varphi)$.

\medskip\noindent\textbf{A conversion that takes no derivative.}
The data of the framed form are produced from the original system by pure
pointwise algebra. One bundles the two real equations into a single complex
equation $E$ with four derivative slots (Section~\ref{sec:slots}), computes the
Beltrami coefficient $\mu$ of the system pointwise from its structure
polynomial (Section~\ref{sec:skeleton}), and replaces $E$ by $E + \mu\,\bar E$;
the result is a framed Beltrami--Vekua equation, with frame, lower-order data,
and forcing all pointwise algebraic in the coefficients of the system
(Theorem~\ref{thm:conversion}). No derivative of any coefficient is taken. In
particular the form exists already for systems whose coefficients are merely
continuous --- and, in Section~\ref{sec:measurable}, merely measurable --- a
class on which the seven-step pipeline of \cite{mass}, which differentiates the
principal coefficients at its second step, does not operate.

\medskip\noindent\textbf{The main result: the pseudo-analytic mass of the
framed form.} The frame carries the single derivative of the whole theory, in
its $L$-Wronskian $W_L(\Phi, \Psi) := \Phi\, L\Psi - \Psi\, L\Phi$. Forming the
\emph{numerator field} $N := \Phi\,\fb - \Psi\,\fa - W_L(\Phi, \Psi)$ and the
2-form
\[
\Theta \;=\; \frac{|N|^2}{\bigl(|\Phi|^2 - |\Psi|^2\bigr)^2\,(1 - |\mu|^2)}\;
dx\, dy
\]
(Definition~\ref{def:mass}), the principal result
(Theorem~\ref{thm:framed-mass}) is that $\Theta$ is invariant under every
recombination of the unknown and covariant under every orientation-preserving
change of variables, so that the total mass
$\Mcal = \int_\Omega \Theta \in [0, \infty]$ is an invariant of the
equivalence class. At the trivial frame the Wronskian vanishes, $N = \fb$, and
$\Theta$ recovers the density of \cite{mass}.

\medskip\noindent\textbf{Why the Wronskian is in the formula.}
The invariance has a single mechanism. Under the recombination, the numerator
obeys the exact law $N' = (|\varphi|^2 - |\psi|^2)\,N$
(equation~\eqref{eq:N-law}), which together with the multiplicativity of the
frame determinant cancels every factor of $\Theta$. The exactness is not a
fortunate collapse among many terms. The recombination contaminates the
lower-order data $(\fa, \fb)$ with derivatives of $(\varphi, \psi)$, and the
$L$-Wronskian, recomputed on the new frame, acquires \emph{exactly the same}
contaminating expression; the two cancel identically. The Wronskian sits in
$N$ precisely so that this cancellation can happen --- it is the unique
correction that absorbs the substitution's derivative defect, and the whole
invariance is that one subtraction.

\medskip\noindent\textbf{One derivative, and where it falls.}
The derivative ledger of the theory closes at one entry. The conversion is
pointwise; the recombination touches the data through derivatives of
$(\varphi, \psi)$, all of which the frame absorbs into $W_L$; and the change of
variables touches the frame through a single multiplicative weight, which the
lower-order data avoid entirely. The only derivative anywhere in the
construction therefore falls on the frame, and the invariant survives wherever
the frame is weakly differentiable. No partial differential equation is solved
anywhere, save the measurable Riemann mapping theorem invoked by the closing
uniformization corollary.

\medskip\noindent\textbf{Measurable regularity and uniformization.}
Because nothing differentiates $\mu$ or the lower-order data, the theory
descends below the smooth class. Section~\ref{sec:measurable} re-runs the
entire construction with the frame in
$\Wcal := W^{1,2}_{\mathrm{loc}} \cap L^\infty_{\mathrm{loc}}$, the Beltrami
coefficient merely measurable and locally elliptic, and the changes of
variables quasiconformal: $\Theta$ is automatically locally integrable, and
every transformation law holds almost everywhere
(Theorem~\ref{thm:measurable-mass}). This is the construction's natural home.
The smooth theory cannot transform its own Beltrami coefficient away --- the
uniformizing map of a non-smooth $\mu$ is in general only
$W^{1,2}_{\mathrm{loc}}$ --- whereas at measurable regularity every framed
equation with $\|\mu\|_\infty < 1$ is quasiconformally equivalent, of equal
mass, to one over $\tilde\mu = 0$: a framed \emph{Vekua} equation
(Corollary~\ref{cor:uniformization}). The measurable class is the smallest
extension of the smooth one closed under its own uniformization, and it is
Bojarski's class \cite{bojarski}, where the Beltrami coefficient is merely
measurable.

\medskip\noindent\textbf{Position among the classical forms.}
The framed form realizes the two classical virtues at once. Like Vekua's
$w_{\bar z} + A w + B\bar w = F$ \cite{vekua} it carries $\bar w$ --- not
$\overline{w_{\bar z}}$ --- in the conjugate slot, and so supports the
phase-only gauge action on which $|\Bcal|^2$ becomes an intrinsic density;
like Bojarski's form \cite{bojarski} it is reached by a purely algebraic
elimination valid at measurable coefficients. It pays neither Vekua's
uniformizing PDE solve nor Bojarski's doubled principal part. The recombination
$w \mapsto \varphi w + \psi\bar w$ that the frame is built to absorb is exactly
the symmetry that the conjugate Vekua coupling admits and that Bojarski's
$\overline{w_{\bar z}}$ coupling does not.

\medskip\noindent\textbf{Relation to \cite{mass} and \cite{abs}.}
One recombination and one scaling carry any framed equation onto the
trivial-frame slice in closed form (Section~\ref{sec:straightening}): the
straightened conjugate coefficient is the numerator field $N$ normalized by the
frame determinant, and $\Theta$ is the Beltrami--Vekua density of \cite{mass}
evaluated on the straightening. That the present invariant agrees, on the
slice, with the output of the seven-step pipeline of \cite{mass} is the
Absorption Theorem of the companion paper \cite{abs}.

\medskip\noindent\textbf{Outline.}
Section~\ref{sec:skeleton} fixes the skeleton form and computes the Beltrami
coefficient $\mu$. Section~\ref{sec:slots} bundles the system into the complex
slot form. Sections~\ref{sec:framed-def}--\ref{sec:conversion} define the
framed equation and convert any elliptic system to it, pointwise. Sections
\ref{sec:substitution} and~\ref{sec:diffeo} compute the closed transformation
laws under recombination of the unknown and under orientation-preserving
changes of variables. Section~\ref{sec:mass} assembles $\Theta$ and proves the
main invariance theorem. Section~\ref{sec:straightening} reduces the framed
form to the Beltrami--Vekua slice. Section~\ref{sec:measurable} re-runs the
whole construction at measurable regularity and closes with the uniformization
corollary that the smooth category does not admit.
\section{The skeleton form}\label{sec:skeleton}

Let $\Omega \subset \R^2 \cong \C$ be a domain with coordinate $z = x + iy$. A first-order real planar elliptic system in real unknowns $(u, v)$, in \emph{skeleton form}, is
\begin{equation}\label{eq:skeleton}
\begin{cases}
-v_y + a_{11}\, u_x + a_{12}\, u_y + a_{13}\, u + a_{14}\, v = f_1, \\[2pt]
\phantom{-}v_x + a_{21}\, u_x + a_{22}\, u_y + a_{23}\, u + a_{24}\, v = f_2,
\end{cases}
\end{equation}
with $a_{ij}, f_k \in C^0(\Omega; \R)$ and the ellipticity conditions
\[
a_{11} > 0, \qquad \delta \;:=\; a_{11}a_{22} - \tfrac14\,(a_{12} + a_{21})^2 \;>\; 0
\]
pointwise. Every first-order $2\times 2$ real elliptic system reaches \eqref{eq:skeleton} by a pointwise row operation \cite[\S 2]{mass}. As in \cite{mass}, set
\begin{equation}\label{eq:structure-data}
\alpha := \frac{a_{22}}{a_{11}}, \qquad \beta := -\frac{a_{12} + a_{21}}{a_{11}}, \qquad \Delta := 4\alpha - \beta^2 = \frac{4\,\delta}{a_{11}^2} > 0,
\end{equation}
and let $\lambda \in C^0(\Omega; \C)$ be the upper-half-plane root of the \emph{structure polynomial},
\begin{equation}\label{eq:structure-poly}
\lambda^2 + \beta\,\lambda + \alpha = 0, \qquad \lambda = \frac{-\beta + i\sqrt{\Delta}}{2} \;=:\; \ell_1 + i\,\ell_2,
\end{equation}
so that $\ell_1 = -\beta/2$, $\ell_2 = \sqrt{\Delta}/2 > 0$, $\lambda + \bar\lambda = -\beta$, $|\lambda|^2 = \alpha$. The \emph{Beltrami coefficient} of the system is its Cayley image
\begin{equation}\label{eq:mu-def}
\mu \;:=\; \frac{\lambda - i}{\lambda + i} \;=\; \frac{(\alpha - 1) + i\beta}{\alpha + 1 + \sqrt{\Delta}}\,,
\qquad
1 - |\mu|^2 \;=\; \frac{2\sqrt{\Delta}}{\alpha + 1 + \sqrt{\Delta}} \;>\; 0,
\end{equation}
where the two closed forms follow from $|\lambda + i|^2 = \alpha + 1 + 2\ell_2 = \alpha + 1 + \sqrt{\Delta}$: the first by $\mu = (\lambda - i)(\bar\lambda - i)/|\lambda + i|^2 = [(|\lambda|^2 - 1) - i(\lambda + \bar\lambda)]/|\lambda+i|^2$, and the second by
\[
|\lambda + i|^4 - \bigl[(\alpha-1)^2 + \beta^2\bigr] = (\alpha+1+\sqrt\Delta)^2 - (\alpha+1)^2 + \Delta = 2\sqrt\Delta\,(\alpha + 1 + \sqrt\Delta)\cdot 1,
\]
using $(\alpha+1)^2 - (\alpha-1)^2 - \beta^2 = 4\alpha - \beta^2 = \Delta$. In particular $\mu \in C^0(\Omega; \D)$, computed pointwise.

\section{The complex slot form}\label{sec:slots}

Set $w := u + iv$ and recall $\overline{w_z} = (\bar w)_{\bar z}$, $\overline{w_{\bar z}} = (\bar w)_z$. The four real first derivatives of $(u, v)$ are real-linearly equivalent to the pair $(w_z, w_{\bar z})$ together with its conjugates, so the complex residual $E := R_1 + iR_2$ of \eqref{eq:skeleton} has a unique presentation with four derivative slots. Expanding $u = (w + \bar w)/2$, $v = -i(w - \bar w)/2$ and collecting:
\begin{equation}\label{eq:slot-form}
E \;=\; s\; w_{\bar z} \;+\; r\; w_z \;+\; r\; \overline{w_{\bar z}} \;+\; (s - 2)\; \overline{w_z} \;+\; a\, w + b\, \bar w \;-\; f,
\end{equation}
where
\begin{equation}\label{eq:rs}
s := \tfrac12\bigl[(a_{11} + a_{22} + 2) + i\,(a_{21} - a_{12})\bigr], \qquad
r := \tfrac12\bigl[(a_{11} - a_{22}) + i\,(a_{12} + a_{21})\bigr],
\end{equation}
\[
a = \tfrac12\bigl[(a_{13} + a_{24}) + i\,(a_{23} - a_{14})\bigr], \qquad
b = \tfrac12\bigl[(a_{13} - a_{24}) + i\,(a_{23} + a_{14})\bigr], \qquad
f = f_1 + i f_2 .
\]
Conjugation swaps the slots pairwise: $\bar E$ carries $\bar s$ on $\overline{w_{\bar z}}$, $\bar r$ on $\overline{w_z}$, $\bar r$ on $w_{\bar z}$, $(\bar s - 2)$ on $w_z$, and $(\bar a, \bar b, \bar f)$ on $(\bar w, w, 1)$. All coefficients are $C^0$ and pointwise algebraic in those of \eqref{eq:skeleton}.

\section{The framed Beltrami--Vekua form}\label{sec:framed-def}

\begin{definition}\label{def:framed}
Let $\mu \in C^0(\Omega; \D)$ and write
\[
L := \bar\partial - \mu\,\partial, \qquad M := \partial - \bar\mu\,\bar\partial
\]
for the Beltrami operator of $\mu$ and its conjugate, so that $\overline{Mw} = \overline{w_z} - \mu\,\overline{w_{\bar z}}$. A \emph{framed Beltrami--Vekua equation} on $\Omega$ is an equation
\begin{equation}\label{eq:framed}
\Phi\, L w \;+\; \Psi\, \overline{M w} \;+\; \fa\, w + \fb\, \bar w \;=\; \ff
\end{equation}
with $\Phi, \Psi, \fa, \fb, \ff \in C^0(\Omega; \C)$ and $|\Phi| > |\Psi|$ pointwise. The pair $(\Phi, \Psi)$ is the \emph{frame}; the Beltrami--Vekua equation of \cite{mass} is the case $(\Phi, \Psi) = (1, 0)$.
\end{definition}

For later use we record the elementary intertwining, valid for any $C^1$ function $g$ and immediate from $\partial_{\bar z}\bar g = \overline{\partial_z g}$:
\begin{equation}\label{eq:intertwine}
L\bar g = \overline{M g}, \qquad M\bar g = \overline{L g}.
\end{equation}

\section{The conversion}\label{sec:conversion}

\begin{theorem}[Conversion to framed form]\label{thm:conversion}
Let \eqref{eq:skeleton} be elliptic with $C^0$ coefficients, $E$ its slot form \eqref{eq:slot-form}, and $\mu$ its Beltrami coefficient \eqref{eq:mu-def}. Then
\[
E + \mu\,\bar E
\]
is a framed Beltrami--Vekua equation \eqref{eq:framed} over $\mu$, with
\begin{equation}\label{eq:conversion-data}
\Phi = s + \mu\,\bar r, \qquad
\Psi = \Phi - 2, \qquad
\fa = a + \mu\,\bar b, \qquad
\fb = b + \mu\,\bar a, \qquad
\ff = f + \mu\,\bar f,
\end{equation}
and frame determinant
\begin{equation}\label{eq:frame-det}
|\Phi|^2 - |\Psi|^2 \;=\; 2\,a_{11}\sqrt{\Delta} \;>\; 0 .
\end{equation}
All data are pointwise algebraic in the coefficients of the system.
\end{theorem}

\begin{proof}
Collecting slots, $E + \mu\bar E$ carries $s + \mu\bar r$ on $w_{\bar z}$, $\;r + \mu(\bar s - 2)$ on $w_z$, $\;r + \mu\bar s$ on $\overline{w_{\bar z}}$, $\;(s-2) + \mu\bar r$ on $\overline{w_z}$, and the lower-order data \eqref{eq:conversion-data}. Matching against the slots of \eqref{eq:framed}, which are $\Phi$ on $w_{\bar z}$, $-\Phi\mu$ on $w_z$, $-\Psi\mu$ on $\overline{w_{\bar z}}$, $\Psi$ on $\overline{w_z}$, the two required identities
\[
r + \mu(\bar s - 2) = -\mu\,(s + \mu\bar r), \qquad
r + \mu\bar s = -\mu\,\bigl((s - 2) + \mu\bar r\bigr)
\]
are one and the same, namely (using $s + \bar s - 2 = a_{11} + a_{22}$)
\begin{equation}\label{eq:key-identity}
r \;+\; \mu\,(a_{11} + a_{22}) \;+\; \mu^2\,\bar r \;=\; 0 .
\end{equation}
To verify \eqref{eq:key-identity}, multiply by $(\lambda + i)^2$ and use $\mu(\lambda + i) = \lambda - i$, so $\mu(\lambda+i)^2 = \lambda^2 + 1$ and $\mu^2(\lambda+i)^2 = (\lambda - i)^2$. Writing $2r = P + iQ$ with $P = a_{11} - a_{22}$ and $Q = a_{12} + a_{21}$,
\begin{align*}
2\,(\lambda+i)^2\,\bigl[\,r + \mu(a_{11}+a_{22}) + \mu^2\bar r\,\bigr]
&= (P + iQ)(\lambda^2 + 2i\lambda - 1) + (P - iQ)(\lambda^2 - 2i\lambda - 1) \\
&\qquad + 2\,(a_{11} + a_{22})(\lambda^2 + 1) \\[2pt]
&= 2P(\lambda^2 - 1) - 4Q\lambda + 2(a_{11} + a_{22})(\lambda^2 + 1) \\[2pt]
&= 4a_{11}\lambda^2 - 4Q\lambda + 4a_{22}
\;=\; 4a_{11}\bigl(\lambda^2 + \beta\lambda + \alpha\bigr) \;=\; 0,
\end{align*}
by $Q = -\beta a_{11}$, $a_{22} = \alpha a_{11}$, and the structure polynomial \eqref{eq:structure-poly}. So the slots of $E + \mu\bar E$ lie in framed position with $\Phi = s + \mu\bar r$ and $\Psi = (s - 2) + \mu\bar r = \Phi - 2$.

For \eqref{eq:frame-det}: since $\Psi = \Phi - 2$,
\[
|\Phi|^2 - |\Psi|^2 = 4\operatorname{Re}\Phi - 4, \qquad
2\operatorname{Re}\Phi = (s + \bar s) + (\mu\bar r + \bar\mu r) = (a_{11} + a_{22} + 2) + 2\operatorname{Re}(\mu\bar r).
\]
By the closed form of $\mu$ in \eqref{eq:mu-def}, with $D := \alpha + 1 + \sqrt{\Delta} = |\lambda + i|^2$,
\[
2\operatorname{Re}(\mu\bar r) = \frac{(\alpha - 1)P + \beta Q}{D}
= \frac{(\alpha-1)\,a_{11}(1 - \alpha) - \beta^2 a_{11}}{D}
= -\,\frac{a_{11}\bigl[(1-\alpha)^2 + \beta^2\bigr]}{D},
\]
so, using $(1+\alpha)^2 - (1-\alpha)^2 - \beta^2 = \Delta$ and $(1+\alpha)D = (1+\alpha)^2 + (1+\alpha)\sqrt\Delta$,
\begin{align*}
4\operatorname{Re}\Phi - 4
&= 2a_{11}(1+\alpha) - \frac{2a_{11}\bigl[(1-\alpha)^2 + \beta^2\bigr]}{D}
= \frac{2a_{11}\bigl[(1+\alpha)\sqrt\Delta + \Delta\bigr]}{D} \\
&= \frac{2a_{11}\sqrt\Delta\,\bigl[\alpha + 1 + \sqrt\Delta\bigr]}{D}
= 2a_{11}\sqrt{\Delta}.
\end{align*}
Positivity is ellipticity, and $|\Phi| > |\Psi|$ follows.
\end{proof}

\begin{remark}\label{rem:cost-zero}
The conversion uses three pointwise-algebraic operations --- the slot bundling, the root $\lambda$, the mix $E + \mu\bar E$ --- and nothing else. In particular it is defined for systems whose coefficients are merely continuous, a class on which the pipeline of \cite{mass} (which differentiates the principal coefficients at its Step 2) does not operate.
\end{remark}

\begin{remark}[The conversion preserves the solution set]\label{rem:invertible}
The passage $E \mapsto E + \mu\,\bar E$ is, at each point, the real-linear
self-map of $\C$
\[
T_\mu(\xi) \;:=\; \xi + \mu\,\bar\xi, \qquad \det T_\mu = 1 - |\mu|^2,
\]
invertible by ellipticity. Its inverse is the same mixing run with $-\mu$,
normalized by the determinant:
\[
T_{-\mu}\circ T_\mu \;=\; (1 - |\mu|^2)\,\mathrm{Id},
\qquad\text{so}\qquad
T_\mu^{-1}(\eta) \;=\; \frac{\eta - \mu\,\bar\eta}{1 - |\mu|^2},
\]
as one checks by expanding
$T_{-\mu}(\xi + \mu\bar\xi) = \xi + \mu\bar\xi - \mu(\bar\xi + \bar\mu\xi)
= (1 - |\mu|^2)\,\xi$. Consequently the framed equation \eqref{eq:framed} of
Theorem~\ref{thm:conversion} has exactly the solutions of the original system:
applying $T_{-\mu}$ to $E + \mu\bar E = 0$ gives $(1 - |\mu|^2)\,E = 0$, hence
$E = 0$, and the converse is immediate. The conversion is thus a genuine normal
form, not merely a derived equation. Like the conversion itself, the inverse is
pointwise, differentiates nothing, and remains valid where $\mu$ is merely
measurable (Section~\ref{sec:measurable}).
\end{remark}

\begin{lemma}[Closed form of the converted frame]\label{lem:frame-closed-form}
For the conversion $E + \mu\bar E$ of Theorem~\ref{thm:conversion}, the frame is
\[
\Phi \;=\; \bigl(1 + \sqrt{\delta}\,\bigr) \;+\; \tfrac{i}{2}\,(a_{21} - a_{12}),
\qquad \Psi = \Phi - 2,
\]
where $\delta = a_{11}a_{22} - \tfrac14(a_{12}+a_{21})^2$ is the ellipticity
discriminant of \eqref{eq:skeleton}. In particular $\Phi$ does not depend on
$\mu$: it is determined by $\delta$ and the antisymmetric part $a_{21} - a_{12}$
alone.
\end{lemma}

\begin{proof}
The cross term $\mu\bar r$ is real. By \eqref{eq:rs} and \eqref{eq:structure-data},
$\bar r = \tfrac{a_{11}}{2}\bigl[(1-\alpha) + i\beta\bigr]$, while
$\mu = \bigl[(\alpha-1)+i\beta\bigr]/D$ with $D = \alpha+1+\sqrt\Delta > 0$ by
\eqref{eq:mu-def}, so
\[
\mu\bar r
= \frac{a_{11}}{2D}\bigl[(\alpha-1)+i\beta\bigr]\bigl[(1-\alpha)+i\beta\bigr]
= -\frac{a_{11}}{2D}\bigl[(\alpha-1)^2 + \beta^2\bigr] \in \R.
\]
Hence $\operatorname{Im}\Phi = \operatorname{Im} s = \tfrac12(a_{21}-a_{12})$ by
\eqref{eq:rs}. For the real part, \eqref{eq:frame-det} gives
$|\Phi|^2 - |\Psi|^2 = 4\operatorname{Re}\Phi - 4 = 2a_{11}\sqrt\Delta$, and
$a_{11}\sqrt\Delta = 2\sqrt\delta$ by \eqref{eq:structure-data}; thus
$\operatorname{Re}\Phi = 1 + \sqrt\delta$.
\end{proof}

\section{The substitution law}\label{sec:substitution}

From here on the frame and the substitutions are $C^1$; the Beltrami coefficient is continuous with $|\mu| < 1$, and the lower-order data remain $C^0$. Throughout, $\mu$ is never differentiated: it enters every identity of this and the following two sections only as a bounded multiplier inside $L = \bar\partial - \mu\,\partial$ and $M = \partial - \bar\mu\,\bar\partial$.

\begin{proposition}[Substitution law]\label{prop:substitution}
The framed class is closed under the recombinations of the unknown
\[
w \;=\; \varphi\, w' + \psi\, \bar w', \qquad \varphi, \psi \in C^1(\Omega; \C), \quad |\varphi| > |\psi| \text{ pointwise},
\]
with the data of \eqref{eq:framed} transforming by
\begin{align}
(\Phi, \Psi) \;&\longmapsto\; \bigl(\Phi\varphi + \Psi\bar\psi,\;\; \Phi\psi + \Psi\bar\varphi\bigr),
\qquad \mu \;\longmapsto\; \mu, \label{eq:frame-law}\\[2pt]
\fa \;&\longmapsto\; \Phi\,(L\varphi) + \Psi\,\overline{(M\psi)} + \fa\,\varphi + \fb\,\bar\psi, \label{eq:a-law}\\
\fb \;&\longmapsto\; \Phi\,(L\psi) + \Psi\,\overline{(M\varphi)} + \fa\,\psi + \fb\,\bar\varphi, \label{eq:b-law}
\end{align}
and $\ff$ unchanged. The law \eqref{eq:frame-law} is the composition of the pointwise real-linear maps $z \mapsto \varphi z + \psi\bar z$; in particular
\begin{equation}\label{eq:det-multiplicative}
|\Phi'|^2 - |\Psi'|^2 \;=\; \bigl(|\Phi|^2 - |\Psi|^2\bigr)\bigl(|\varphi|^2 - |\psi|^2\bigr),
\end{equation}
so the frame condition $|\Phi| > |\Psi|$ is preserved and the class is closed.
\end{proposition}

\begin{proof}
By the product rule and the intertwining \eqref{eq:intertwine},
\[
L(\varphi w') = \varphi\, Lw' + (L\varphi)\, w', \qquad
L(\psi\bar w') = \psi\,\overline{Mw'} + (L\psi)\,\bar w',
\]
\[
M(\varphi w') = \varphi\, Mw' + (M\varphi)\, w', \qquad
M(\psi\bar w') = \psi\,\overline{Lw'} + (M\psi)\,\bar w' .
\]
Substituting $w = \varphi w' + \psi\bar w'$ into \eqref{eq:framed}, conjugating the $M$-lines, and collecting the coefficients of $Lw'$, $\overline{Mw'}$, $w'$, $\bar w'$ yields \eqref{eq:frame-law}--\eqref{eq:b-law}. Identity \eqref{eq:det-multiplicative} is the multiplicativity of determinants of real-linear maps of $\C$, by direct expansion:
$|\Phi\varphi + \Psi\bar\psi|^2 - |\Phi\psi + \Psi\bar\varphi|^2 = (|\Phi|^2 - |\Psi|^2)(|\varphi|^2 - |\psi|^2)$, the mixed terms $2\operatorname{Re}(\Phi\varphi\bar\Psi\psi)$ cancelling between the two squares.
\end{proof}

We also record the trivial recombination of the equation itself, $E \mapsto c\,E$ with $c \in C^0(\Omega; \C^*)$, under which
\begin{equation}\label{eq:row-scaling}
(\Phi, \Psi, \fa, \fb, \ff) \;\longmapsto\; c\,(\Phi, \Psi, \fa, \fb, \ff), \qquad \mu \;\longmapsto\; \mu .
\end{equation}

\section{The diffeomorphism law}\label{sec:diffeo}

\begin{proposition}[Diffeomorphism law]\label{prop:diffeo}
Let $F : \Omega \to \Omega'$ be an orientation-preserving $C^1$ diffeomorphism, $J := |F_z|^2 - |F_{\bar z}|^2 > 0$ its Jacobian, and let $w$ satisfy the framed equation \eqref{eq:framed} on $\Omega'$. Then $h := w \circ F$ satisfies on $\Omega$ the framed equation with data
\begin{equation}\label{eq:diffeo-law}
\tilde\mu \;=\; \frac{F_{\bar z} + (\mu\circ F)\,\overline{F_z}}{F_z + (\mu\circ F)\,\overline{F_{\bar z}}}, \qquad
(\tilde\Phi, \tilde\Psi) = \rho\,\bigl(\Phi\circ F,\; \Psi\circ F\bigr), \qquad
(\tilde\fa, \tilde\fb, \tilde\ff) = \bigl(\fa, \fb, \ff\bigr)\circ F,
\end{equation}
where
\[
\rho \;:=\; \frac{F_z + (\mu\circ F)\,\overline{F_{\bar z}}}{J} \;\neq\; 0 .
\]
Moreover $|\tilde\mu| < 1$, with
\begin{equation}\label{eq:mu-tilde-det}
1 - |\tilde\mu|^2 \;=\; \frac{1 - |\mu\circ F|^2}{|\rho|^2\, J}\,,
\end{equation}
and $|\tilde\Phi| > |\tilde\Psi|$; the class is closed. The law touches no derivative of the data: the frame picks up the single weight $\rho$, and the lower-order data simply compose.
\end{proposition}

\begin{proof}
Viewing $w$ as a function of $(\zeta, \bar\zeta)$ on $\Omega'$, the chain rule reads
\[
h_z = F_z\,(w_\zeta\circ F) + \overline{F_{\bar z}}\,(w_{\bar\zeta}\circ F), \qquad
h_{\bar z} = F_{\bar z}\,(w_\zeta\circ F) + \overline{F_z}\,(w_{\bar\zeta}\circ F),
\]
which inverts, with determinant $J$, to
\[
w_\zeta\circ F = \frac{\overline{F_z}\, h_z - \overline{F_{\bar z}}\, h_{\bar z}}{J}, \qquad
w_{\bar\zeta}\circ F = \frac{F_z\, h_{\bar z} - F_{\bar z}\, h_z}{J} .
\]
Hence
\[
(Lw)\circ F
= \frac{\bigl(F_z + (\mu\circ F)\overline{F_{\bar z}}\bigr)\,h_{\bar z} - \bigl(F_{\bar z} + (\mu\circ F)\overline{F_z}\bigr)\,h_z}{J}
\;=\; \rho\,\bigl(h_{\bar z} - \tilde\mu\, h_z\bigr) \;=\; \rho\,\tilde L h,
\]
with $\tilde L := \bar\partial - \tilde\mu\,\partial$, and by the conjugate computation $(Mw)\circ F = \bar\rho\,\tilde M h$, so $\overline{Mw}\circ F = \rho\,\overline{\tilde M h}$. Composing \eqref{eq:framed} with $F$ and substituting,
\[
(\Phi\circ F)\,\rho\,\tilde L h \;+\; (\Psi\circ F)\,\rho\,\overline{\tilde M h} \;+\; (\fa\circ F)\, h + (\fb\circ F)\, \bar h \;=\; \ff\circ F,
\]
which is \eqref{eq:diffeo-law}. For \eqref{eq:mu-tilde-det}: expanding numerator and denominator of $1 - |\tilde\mu|^2$,
\[
\bigl|F_z + (\mu\circ F)\overline{F_{\bar z}}\bigr|^2 - \bigl|F_{\bar z} + (\mu\circ F)\overline{F_z}\bigr|^2
= \bigl(1 - |\mu\circ F|^2\bigr)\,\bigl(|F_z|^2 - |F_{\bar z}|^2\bigr),
\]
the cross terms $2\operatorname{Re}\bigl((\mu\circ F)\overline{F_{\bar z}}\,\overline{F_z}\bigr)$ cancelling between the two squares; dividing by $|F_z + (\mu\circ F)\overline{F_{\bar z}}|^2 = |\rho|^2 J^2$ gives \eqref{eq:mu-tilde-det}. In particular the left side of \eqref{eq:mu-tilde-det} is positive, so $|\tilde\mu| < 1$; and $\rho \neq 0$ since $|\mu\circ F|\,|\overline{F_{\bar z}}| < |F_z|$. Finally $|\tilde\Phi| > |\tilde\Psi|$ because the frame scales by the common factor $\rho$.
\end{proof}

\section{The pseudo-analytic mass of the framed form}\label{sec:mass}

\begin{definition}\label{def:mass}
For a framed equation \eqref{eq:framed} with $C^1$ frame and continuous Beltrami coefficient ($|\mu| < 1$), define the \emph{$L$-Wronskian} of the frame and the \emph{numerator field}
\[
W_L(\Phi, \Psi) \;:=\; \Phi\, L\Psi \;-\; \Psi\, L\Phi, \qquad
N \;:=\; \Phi\,\fb \;-\; \Psi\,\fa \;-\; W_L(\Phi, \Psi),
\]
and the 2-form and total mass --- the \emph{pseudo-analytic mass} of the framed equation ---
\begin{equation}\label{eq:mass-def}
\Theta \;:=\; \frac{|N|^2}{\bigl(|\Phi|^2 - |\Psi|^2\bigr)^2\,\bigl(1 - |\mu|^2\bigr)}\; dx\, dy,
\qquad
\Mcal \;:=\; \int_\Omega \Theta \;\in\; [0, \infty] .
\end{equation}
\end{definition}

\begin{theorem}[Pseudo-analytic mass of the framed form]\label{thm:framed-mass}
The form $\Theta$ is:
\begin{itemize}
\item[(i)] invariant under every substitution $w = \varphi w' + \psi\bar w'$, $\varphi, \psi \in C^1$, $|\varphi| > |\psi|$, acting by \eqref{eq:frame-law}--\eqref{eq:b-law}; indeed the numerator obeys the exact law
\begin{equation}\label{eq:N-law}
N' \;=\; \bigl(|\varphi|^2 - |\psi|^2\bigr)\, N ;
\end{equation}
\item[(ii)] invariant under the scalings \eqref{eq:row-scaling}: $N' = c^2 N$;
\item[(iii)] equal, at the trivial frame $(\Phi, \Psi) = (1, 0)$, to the density $|\fb|^2\,(1 - |\mu|^2)^{-1}\,dx\,dy$ of the Beltrami--Vekua equation \cite{mass};
\item[(iv)] covariant under every orientation-preserving $C^1$ diffeomorphism $F$, acting by \eqref{eq:diffeo-law}: $\tilde\Theta = F^*\Theta$, and hence $\Mcal$ is unchanged.
\end{itemize}
\end{theorem}

\begin{proof}
(i) Write $D := |\varphi|^2 - |\psi|^2$. We compute $W_L(\Phi', \Psi')$ for the new frame \eqref{eq:frame-law}. By the product rule and \eqref{eq:intertwine} (which gives $L\bar\varphi = \overline{M\varphi}$, $L\bar\psi = \overline{M\psi}$),
\[
L\Psi' = (L\Phi)\,\psi + \Phi\,(L\psi) + (L\Psi)\,\bar\varphi + \Psi\,\overline{M\varphi}, \qquad
L\Phi' = (L\Phi)\,\varphi + \Phi\,(L\varphi) + (L\Psi)\,\bar\psi + \Psi\,\overline{M\psi}.
\]
In $W_L(\Phi', \Psi') = \Phi' L\Psi' - \Psi' L\Phi'$, the terms carrying derivatives of the \emph{old frame} collect as
\[
(L\Phi)\bigl[(\Phi\varphi + \Psi\bar\psi)\psi - (\Phi\psi + \Psi\bar\varphi)\varphi\bigr]
+ (L\Psi)\bigl[(\Phi\varphi + \Psi\bar\psi)\bar\varphi - (\Phi\psi + \Psi\bar\varphi)\bar\psi\bigr]
= D\,\bigl(\Phi\,L\Psi - \Psi\,L\Phi\bigr),
\]
since the first bracket is $-\Psi D$ and the second is $\Phi D$. The terms carrying derivatives of the \emph{substitution} collect as
\begin{equation}\label{eq:contamination}
\Phi'\bigl[\Phi(L\psi) + \Psi\overline{M\varphi}\bigr] \;-\; \Psi'\bigl[\Phi(L\varphi) + \Psi\overline{M\psi}\bigr] .
\end{equation}
Therefore
\[
W_L(\Phi', \Psi') \;=\; D\,W_L(\Phi, \Psi) \;+\; \eqref{eq:contamination}.
\]
On the other hand, by \eqref{eq:a-law}--\eqref{eq:b-law},
\[
\Phi'\fb' - \Psi'\fa'
= \underbrace{\Phi'\bigl[\Phi(L\psi) + \Psi\overline{M\varphi}\bigr] - \Psi'\bigl[\Phi(L\varphi) + \Psi\overline{M\psi}\bigr]}_{=\;\eqref{eq:contamination}}
\;+\; \fa\bigl[\Phi'\psi - \Psi'\varphi\bigr] + \fb\bigl[\Phi'\bar\varphi - \Psi'\bar\psi\bigr],
\]
and the same two brackets as before give $\Phi'\psi - \Psi'\varphi = -\Psi D$ and $\Phi'\bar\varphi - \Psi'\bar\psi = \Phi D$. Subtracting, the contamination \eqref{eq:contamination} cancels identically:
\[
N' \;=\; \Phi'\fb' - \Psi'\fa' - W_L(\Phi', \Psi')
\;=\; D\,\bigl(\Phi\fb - \Psi\fa\bigr) - D\,W_L(\Phi, \Psi)
\;=\; D\,N,
\]
which is \eqref{eq:N-law}. Then by \eqref{eq:det-multiplicative}, and since $\mu$ and $dx\,dy$ do not move,
\[
\Theta' = \frac{D^2\,|N|^2}{D^2\,(|\Phi|^2 - |\Psi|^2)^2\,(1 - |\mu|^2)}\,dx\,dy = \Theta .
\]

(ii) $\Phi'\fb' - \Psi'\fa' = c^2(\Phi\fb - \Psi\fa)$, and in $W_L(c\Phi, c\Psi)$ the Leibniz cross terms carry the factor $\Phi\Psi - \Psi\Phi = 0$, so $W_L(c\Phi, c\Psi) = c^2\,W_L(\Phi, \Psi)$ and $N' = c^2 N$; the numerator of $\Theta$ scales by $|c|^4$ and the denominator by $|c|^4$.

(iii) Immediate: $W_L(1, 0) = 0$ and $N = \fb$.

(iv) The pulled-back frame $\rho\,(\Phi\circ F, \Psi\circ F)$ need not be $C^1$, as $\rho$ is built from the first derivatives of $F$ alone; we therefore read $\Theta$ off the equivalent representative produced by the row scaling \eqref{eq:row-scaling} with $c = \rho^{-1}$, in which $\rho$ cancels and the frame $(\Phi\circ F, \Psi\circ F)$ is again $C^1$, with data
\[
\tilde\mu, \qquad (\Phi\circ F,\; \Psi\circ F), \qquad \rho^{-1}\,(\fa, \fb, \ff)\circ F .
\]
Any two $C^1$-framed scalings of the pulled-back equation differ by a $C^1$ factor bounded below, so by (ii) the value is independent of the choice. For any first-order operator $X$ write $W_X(P, Q) := P\,XQ - Q\,XP$, so that $W_L = W_{\bar\partial} - \mu\, W_{\partial}$. The chain rule for $\tilde L = \bar\partial - \tilde\mu\,\partial$ applied to a composite $g \circ F$ reads
\[
\tilde L(g\circ F) \;=\; \kappa_1\,\bigl(\partial g\circ F\bigr) \;+\; \kappa_2\,\bigl(\bar\partial g\circ F\bigr),
\qquad \kappa_1 := F_{\bar z} - \tilde\mu\, F_z, \quad \kappa_2 := \overline{F_z} - \tilde\mu\,\overline{F_{\bar z}},
\]
and substituting $\tilde\mu$ from \eqref{eq:diffeo-law} gives, by the same cross-term cancellation as in \eqref{eq:mu-tilde-det},
\[
\kappa_1 = -\,\frac{\mu\circ F}{\rho}, \qquad \kappa_2 = \frac{1}{\rho} .
\]
Hence
\[
W_{\tilde L}\bigl(\Phi\circ F,\, \Psi\circ F\bigr)
= \kappa_1\,\bigl(W_{\partial}(\Phi, \Psi)\circ F\bigr) + \kappa_2\,\bigl(W_{\bar\partial}(\Phi, \Psi)\circ F\bigr)
= \frac{1}{\rho}\,\Bigl(\bigl[W_{\bar\partial} - \mu\,W_{\partial}\bigr](\Phi, \Psi)\Bigr)\circ F
= \frac{W_L(\Phi, \Psi)\circ F}{\rho},
\]
and the numerator field of the scaled representative is
\[
\tilde N = (\Phi\circ F)\,\frac{\fb\circ F}{\rho} - (\Psi\circ F)\,\frac{\fa\circ F}{\rho} - \frac{W_L(\Phi,\Psi)\circ F}{\rho}
\;=\; \frac{N\circ F}{\rho}\,.
\]
Assembling with \eqref{eq:mu-tilde-det},
\[
\tilde\Theta
= \frac{|\rho|^{-2}\,\bigl(|N|^2\circ F\bigr)}
{\bigl((|\Phi|^2 - |\Psi|^2)^2\circ F\bigr)\;\dfrac{(1 - |\mu|^2)\circ F}{|\rho|^2\, J}}\;dx\,dy
= \left(\frac{|N|^2}{(|\Phi|^2 - |\Psi|^2)^2\,(1 - |\mu|^2)}\right)\!\circ F\;\; J\; dx\, dy
\;=\; F^*\Theta,
\]
since $F^*(d\xi\,d\eta) = J\,dx\,dy$. Integrating, $\Mcal$ is unchanged.
\end{proof}

\begin{corollary}\label{cor:system-invariant}
Let \eqref{eq:skeleton} have $C^1$ principal and $C^0$ lower-order coefficients, and let $\Theta$ be computed from its conversion $E + \mu\bar E$ via \eqref{eq:conversion-data}, i.e.
\[
\Theta \;=\; \frac{\bigl|\,\Phi\,\fb - \Psi\,\fa - W_L(\Phi, \Psi)\,\bigr|^2}{4\,a_{11}^2\,\Delta\,\bigl(1 - |\mu|^2\bigr)}\; dx\,dy,
\qquad \Phi = s + \mu\bar r, \quad \Psi = \Phi - 2 .
\]
Then $\Mcal = \int_\Omega \Theta$ is unchanged by every $C^1$ recombination of the unknowns $w = \varphi w' + \psi\bar w'$ with $|\varphi| > |\psi|$, by every recombination of the equations, and by every orientation-preserving $C^1$ change of variables.
\end{corollary}

\begin{proof}
Substitutions and changes of variables: Theorem~\ref{thm:framed-mass}(i),(iv), noting that a substitution or pullback of the system commutes with the bundling and changes its converted framed form by the laws \eqref{eq:frame-law}--\eqref{eq:b-law}, resp.\ \eqref{eq:diffeo-law}, up to a scaling \eqref{eq:row-scaling}. Recombinations of the equations preserve the system; its converted framed forms then differ by a scaling, since they share $\mu$ and the framed slot positions determine the frame up to the common factor; and scalings are Theorem~\ref{thm:framed-mass}(ii).
\end{proof}

\begin{remark}\label{rem:consistency}
At the trivial frame, Theorem~\ref{thm:framed-mass}(iii) identifies $\Theta$ with the density of the pseudo-analytic mass of \cite{mass}; that the invariant of Corollary~\ref{cor:system-invariant} coincides with the output of the seven-step pipeline of \cite{mass} is then the Absorption Theorem of \cite{abs}. The two derivative-bearing terms of the theory deserve a final glance: the substitution's contamination of the lower-order data --- the bracket \eqref{eq:contamination} --- and the transformation defect of the Wronskian are \emph{the same expression}, and the entire invariance (i) is their cancellation. The Wronskian is in the formula so that this can happen.
\end{remark}

\section{The straightening}\label{sec:straightening}

The Beltrami--Vekua class of \cite{mass} is the slice $(\Phi, \Psi) = (1, 0)$ of the framed class. One substitution and one scaling carry any framed equation onto it, in closed form; evaluating the result against Theorem~\ref{thm:framed-mass}(iii) shows that $\Theta$ is exactly the Beltrami--Vekua density of the straightened equation. The mass does not change under straightening.

\begin{proposition}[Straightening]\label{prop:straighten}
Let \eqref{eq:framed} have $C^1$ frame and continuous Beltrami coefficient ($|\mu| < 1$), and set
\[
\nu \;:=\; \Psi/\Phi \;\in\; \D
\]
pointwise ($|\nu| < 1$ is the frame condition). The substitution $w = w' - \nu\,\bar w'$ followed by the scaling \eqref{eq:row-scaling} with $c = \bigl(\Phi\,(1 - |\nu|^2)\bigr)^{-1}$ carries the equation to the trivial frame $(1, 0)$ with $\mu$ unchanged: the Beltrami--Vekua equation
\[
w'_{\bar z} \;-\; \mu\, w'_z \;+\; \Acal_{\mathrm{str}}\, w' \;+\; \Bcal_{\mathrm{str}}\, \bar w' \;=\; \Fcal_{\mathrm{str}},
\]
with
\begin{equation}\label{eq:straightened}
\Acal_{\mathrm{str}} \;=\; \frac{\fa - \bar\nu\,\fb - \Psi\,\overline{M\nu}}{\Phi\,(1 - |\nu|^2)}, \qquad
\Bcal_{\mathrm{str}} \;=\; \frac{\fb - \nu\,\fa - \Phi\,L\nu}{\Phi\,(1 - |\nu|^2)}, \qquad
\Fcal_{\mathrm{str}} \;=\; \frac{\ff}{\Phi\,(1 - |\nu|^2)} .
\end{equation}
Moreover, since $L$ is a derivation,
\begin{equation}\label{eq:quotient}
\Phi^2\, L\nu \;=\; \Phi^2\, L\bigl(\Psi/\Phi\bigr) \;=\; \Phi\,L\Psi - \Psi\,L\Phi \;=\; W_L(\Phi, \Psi),
\end{equation}
and therefore
\begin{equation}\label{eq:Bstr-N}
\Bcal_{\mathrm{str}} \;=\; \frac{N}{\Phi^2\,(1 - |\nu|^2)}\,, \qquad
\bigl|\Phi^2(1 - |\nu|^2)\bigr| \;=\; |\Phi|^2 - |\Psi|^2 .
\end{equation}
Any two substitutions landing on the slice differ by a $\C$-linear substitution: from a frame $(\Phi, 0)$, the law \eqref{eq:frame-law} reads $(\Phi\varphi, \Phi\psi)$, so $\Psi' = 0$ forces $\psi = 0$ --- the gauge action of \cite{mass}.
\end{proposition}

\begin{proof}
The frame law \eqref{eq:frame-law} at $(\varphi, \psi) = (1, -\nu)$, with $\Psi = \Phi\nu$:
\[
\bigl(\Phi\cdot 1 + \Psi\,(-\bar\nu),\;\; \Phi\,(-\nu) + \Psi\cdot 1\bigr)
\;=\; \bigl(\Phi\,(1 - |\nu|^2),\;\; 0\bigr) .
\]
The lower-order laws \eqref{eq:a-law}--\eqref{eq:b-law} at $(1, -\nu)$, with $L(1) = M(1) = 0$:
\[
\fa' \;=\; \Phi\,L(1) + \Psi\,\overline{M(-\nu)} + \fa\cdot 1 + \fb\,(-\bar\nu)
\;=\; \fa - \bar\nu\,\fb - \Psi\,\overline{M\nu},
\]
\[
\fb' \;=\; \Phi\,L(-\nu) + \Psi\,\overline{M(1)} + \fa\,(-\nu) + \fb\cdot 1
\;=\; \fb - \nu\,\fa - \Phi\,L\nu .
\]
The scaling \eqref{eq:row-scaling} by $c = 1/\bigl(\Phi(1-|\nu|^2)\bigr)$ normalizes the frame to $(1,0)$ and divides $(\fa', \fb', \ff)$ by $\Phi(1-|\nu|^2)$, giving \eqref{eq:straightened}. Identity \eqref{eq:quotient} is the Leibniz rule for the first-order operator $L$ applied to $\Psi = \nu\,\Phi$: $L\Psi = \nu\,L\Phi + \Phi\,L\nu$, so $\Phi^2 L\nu = \Phi\,L\Psi - \Psi\,L\Phi$. Substituting $\Phi\,L\nu = W_L(\Phi,\Psi)/\Phi$ into \eqref{eq:straightened} gives the first half of \eqref{eq:Bstr-N}; the second half is $|\Phi|^2\,(1 - |\nu|^2) = |\Phi|^2 - |\Psi|^2$. Uniqueness is the displayed frame-law computation.
\end{proof}

\begin{corollary}[The mass is invariant under straightening]\label{cor:straighten-mass}
The straightened equation has trivial frame, so by Theorem~\ref{thm:framed-mass}(iii) its density is $|\Bcal_{\mathrm{str}}|^2\,(1 - |\mu|^2)^{-1}\,dx\,dy$; by \eqref{eq:Bstr-N},
\[
\frac{|\Bcal_{\mathrm{str}}|^2}{1 - |\mu|^2}\; dx\, dy
\;=\; \frac{|N|^2}{\bigl(|\Phi|^2 - |\Psi|^2\bigr)^2\,\bigl(1 - |\mu|^2\bigr)}\; dx\, dy
\;=\; \Theta .
\]
Hence $\Theta$ is the Beltrami--Vekua density of \cite{mass} evaluated on the straightening, $\Mcal$ is the pseudo-analytic mass of the straightened equation, and every framed equation is the substitution image of a Beltrami--Vekua equation with the same $\mu$.
\end{corollary}

\begin{remark}\label{rem:straighten}
The corollary is also forced by Theorem~\ref{thm:framed-mass}(i)--(ii), since the straightening is a substitution followed by a scaling; the point of the display is the closed form: the straightened $\Bcal$ \emph{is} the numerator field $N$, normalized by the frame determinant. The straightening is the algebraic half of the comparison with \cite{mass}: the seven-step pipeline is a straightening computed in the spectral frame $w = U + \lambda V$, and the agreement of the two, up to the universal gauge $-i\lambda/(\varphi - \psi)$, is the Absorption Theorem of \cite{abs}. The derivative ledger also closes here: the conversion of Section~\ref{sec:conversion} is pointwise, and the single differentiation of the data that the Beltrami--Vekua slice costs is the term $L\nu$ of \eqref{eq:straightened} --- equivalently, the Wronskian $W_L$ of \eqref{eq:Bstr-N}.
\end{remark}
\section{The mass at measurable regularity}\label{sec:measurable}

The only derivative in the theory is the one inside the $L$-Wronskian. This section records the weakest setting in which that derivative exists as a locally integrable function and every computation of Sections~\ref{sec:substitution}--\ref{sec:mass} remains valid almost everywhere. One cannot ask for less: $N$ enters $\Theta$ squared, and a genuinely distributional $N$ admits no $|N|^2$; \emph{measurable regularity} therefore means measurable data whose frame possesses weak first derivatives, defined almost everywhere. Two preliminary observations fix the classes. First, no derivative of $\mu$, $\fa$, $\fb$ or $\ff$ is taken anywhere in Sections~\ref{sec:substitution}--\ref{sec:mass}: $\mu$ enters every identity there only as a bounded multiplier inside $L$ and $M$. The descent below the smooth class therefore weakens two data independently --- the frame from $C^1$ to $\Wcal$ and the Beltrami coefficient from continuous to measurable --- with the lower-order data passing from $C^0$ to $L^2_{\mathrm{loc}}$. Second, among Sobolev exponents for the frame, $2$ is selected by the change-of-variables law (Lemma~\ref{lem:toolbox}(b)): quasiconformal composition preserves precisely the Dirichlet class $W^{1,2}_{\mathrm{loc}}$. Write
\[
\Wcal \;:=\; W^{1,2}_{\mathrm{loc}}(\Omega) \,\cap\, L^\infty_{\mathrm{loc}}(\Omega).
\]

\begin{definition}[Measurable framed equation]\label{def:measurable}
A \emph{measurable framed Beltrami--Vekua equation} on $\Omega$ is an equation \eqref{eq:framed} whose data satisfy, for every compact $K \subset \Omega$ and constants $k_K < 1$, $\varepsilon_K > 0$ depending on $K$:
\[
\mu \in L^\infty(\Omega), \quad \operatorname*{ess\,sup}_K\, |\mu| \le k_K; \qquad
\Phi, \Psi \in \Wcal, \quad |\Phi|^2 - |\Psi|^2 \ge \varepsilon_K \ \text{a.e.\ on } K; \qquad
\fa, \fb, \ff \in L^2_{\mathrm{loc}}.
\]
\end{definition}

The data of Definition~\ref{def:measurable} feed the formulas of Definition~\ref{def:mass} without further ado:
\[
L\Phi,\; L\Psi \,\in\, L^2_{\mathrm{loc}}
\;\Longrightarrow\;
W_L(\Phi, \Psi) \,\in\, L^2_{\mathrm{loc}}
\;\Longrightarrow\;
N \;=\; \Phi\fb - \Psi\fa - W_L(\Phi,\Psi) \,\in\, L^2_{\mathrm{loc}},
\]
the first arrow because $\mu \in L^\infty$ multiplies $\Phi_z, \Psi_z \in L^2_{\mathrm{loc}}$, the second and third because $\Phi, \Psi \in L^\infty_{\mathrm{loc}}$; and since the weight $\bigl[(|\Phi|^2 - |\Psi|^2)^2(1 - |\mu|^2)\bigr]^{-1}$ is essentially bounded on $K$ by $\varepsilon_K^{-2}(1 - k_K^2)^{-1}$,
\[
\Theta \,\in\, L^1_{\mathrm{loc}}, \qquad \Mcal \;=\; \int_\Omega \Theta \;\in\; [0, \infty].
\]
No integrability hypothesis is imposed and none is needed: the mass density is automatically locally integrable, and only the total mass can diverge.

\begin{lemma}[Toolbox]\label{lem:toolbox}
\emph{(a)} $\Wcal$ is an algebra: for $f, g \in \Wcal$, $fg \in \Wcal$ with $\partial(fg) = f\,\partial g + g\,\partial f$ and $\bar\partial(fg) = f\,\bar\partial g + g\,\bar\partial f$ a.e.; moreover $\bar\partial\bar g = \overline{\partial g}$ a.e., so the intertwining \eqref{eq:intertwine} holds a.e.\ on $\Wcal$. If in addition $|f| \ge \varepsilon_K$ a.e.\ on compacts, then $1/f \in \Wcal$ with $\partial(1/f) = -f^{-2}\,\partial f$ a.e.

\emph{(b)} Let $F : \Omega \to \Omega'$ be an orientation-preserving homeomorphism in $W^{1,2}_{\mathrm{loc}}$ with $|F_{\bar z}| \le k_F\,|F_z|$ a.e.\ for some $k_F < 1$ (quasiconformal). Then $F$ is differentiable a.e.\ with $J = |F_z|^2 - |F_{\bar z}|^2 > 0$ a.e., satisfies Luzin's conditions (N) and (N$^{-1}$), and the area formula $\int_{\Omega'} g\; d\xi\,d\eta = \int_\Omega (g \circ F)\, J\; dx\,dy$ holds for measurable $g \ge 0$; the inverse $F^{-1}$ is quasiconformal; and composition with $F$ preserves $\Wcal$, with the chain rule
\[
(g \circ F)_z = (g_\zeta \circ F)\, F_z + (g_{\bar\zeta} \circ F)\, \overline{F_{\bar z}}, \qquad
(g \circ F)_{\bar z} = (g_\zeta \circ F)\, F_{\bar z} + (g_{\bar\zeta} \circ F)\, \overline{F_z}
\]
valid a.e.\ for $g \in \Wcal$. All of this is classical; see \cite{aim}.
\end{lemma}

\begin{theorem}[Mass at measurable regularity]\label{thm:measurable-mass}
Let \eqref{eq:framed} be measurable framed. Then $\Theta \in L^1_{\mathrm{loc}}$, and:
\begin{itemize}
\item[(i)] under every substitution $w = \varphi w' + \psi\bar w'$ with $\varphi, \psi \in \Wcal$ and $|\varphi|^2 - |\psi|^2 \ge \varepsilon_K$ a.e.\ on compacts, the laws \eqref{eq:frame-law}--\eqref{eq:b-law} hold a.e., the measurable framed class is closed, and $N' = (|\varphi|^2 - |\psi|^2)\,N$ a.e.; hence $\Theta' = \Theta$ a.e.;
\item[(ii)] under the scalings \eqref{eq:row-scaling} with $c \in \Wcal$, $|c| \ge \varepsilon_K$ a.e.\ on compacts, the class is closed and $N' = c^2 N$ a.e.;
\item[(iii)] at the trivial frame $(\Phi, \Psi) = (1, 0)$, $\Theta$ equals the density $|\fb|^2(1 - |\mu|^2)^{-1}\,dx\,dy$ of \cite{mass}, now defined for $\fb \in L^2_{\mathrm{loc}}$ and $\mu$ merely measurable and locally elliptic;
\item[(iv)] under every orientation-preserving quasiconformal homeomorphism $F$,
the law \eqref{eq:diffeo-law} holds a.e.; the raw pullback, after the row scaling
\eqref{eq:row-scaling} with $c = \rho^{-1}$ --- its \emph{canonical pullback} ---
is measurable framed, and its density satisfies $\tilde\Theta = F^*\Theta$; in
particular $\Mcal$ is unchanged.
\end{itemize}
\end{theorem}

\begin{proof}
Every identity in the proofs of Propositions~\ref{prop:substitution} and \ref{prop:diffeo} and of Theorem~\ref{thm:framed-mass} is a finite composition of pointwise algebra, the Leibniz rule, the intertwining \eqref{eq:intertwine}, and --- in the change of variables --- the chain rule and the area formula. By Lemma~\ref{lem:toolbox} each ingredient holds almost everywhere on the stated classes, so each identity holds almost everywhere; no derivative of $\mu$ or of the lower-order data appears in any of them. Three points require checking.

\emph{Class closure in (i) and (ii).} The new frame \eqref{eq:frame-law} lies in $\Wcal$ by the algebra property and is uniformly framed by \eqref{eq:det-multiplicative}: $|\Phi'|^2 - |\Psi'|^2 \ge \varepsilon_K\,\varepsilon'_K$ a.e.\ on $K$. The new lower-order data \eqref{eq:a-law}--\eqref{eq:b-law} are sums of products of an $L^\infty_{\mathrm{loc}}$ and an $L^2_{\mathrm{loc}}$ function, hence lie in $L^2_{\mathrm{loc}}$. The same applies to the scaled data in (ii).

\emph{Local ellipticity of $\tilde\mu$ in (iv).} On a compact $K$, with $k := \operatorname*{ess\,sup}_{F(K)} |\mu|$ and $|F_{\bar z}| \le k_F\,|F_z|$ a.e.,
\[
|\rho|^2 J \;=\; \frac{\bigl|F_z + (\mu\circ F)\,\overline{F_{\bar z}}\bigr|^2}{J}
\;\le\; \frac{(1 + k\,k_F)^2\,|F_z|^2}{(1 - k_F^2)\,|F_z|^2},
\]
so by \eqref{eq:mu-tilde-det}
\[
1 - |\tilde\mu|^2 \;=\; \frac{1 - |\mu\circ F|^2}{|\rho|^2\,J}
\;\ge\; \frac{(1 - k^2)(1 - k_F^2)}{(1 + k\,k_F)^2} \;>\; 0
\qquad \text{a.e.\ on } K,
\]
which is the required local ellipticity of $\tilde\mu$.

\emph{The weight $\rho$ in (iv).} The raw pulled-back frame
$\rho\,(\Phi\circ F, \Psi\circ F)$ carries the merely measurable factor $\rho$ and
so leaves $\Wcal$: the raw pullback is \emph{not} itself measurable framed, and its
mass density is not given directly by Definition~\ref{def:mass}. We therefore fix the
\emph{canonical pullback} to be the row scaling \eqref{eq:row-scaling} of the raw
pullback by $c = \rho^{-1}$ --- an admissible operation on the equation for any
nonzero measurable $c$ --- under which the weight $\rho$ cancels and the frame
returns to $(\Phi\circ F, \Psi\circ F)$, and read $\Theta$ off this representative.
(No appeal to independence of the scaling is made here; that the value does not
depend on the choice is recorded separately, after the theorem, as
Proposition~\ref{prop:admissible}.) The canonical pullback is measurable framed: its
frame $(\Phi\circ F, \Psi\circ F)$ lies in $\Wcal$ by Lemma~\ref{lem:toolbox}(b) and
is uniformly framed, since $F$ maps compacts to compacts and, by (N) and (N$^{-1}$),
null sets to null sets in both directions; and its lower-order data
$\rho^{-1}(\fa, \fb, \ff)\circ F$ are locally square-integrable, by
$|\rho|^{-1} \le |F_z|\,(1 - k\,k_F)^{-1}$ a.e.\ and the area formula:
\[
\int_K \bigl|\rho^{-1}\,(\fa\circ F)\bigr|^2
\;\le\; \frac{1}{(1 - k\,k_F)^2}\int_K |F_z|^2\,\bigl(|\fa|^2\circ F\bigr)
\;\le\; \frac{(1 - k_F^2)^{-1}}{(1 - k\,k_F)^2}\int_{F(K)} |\fa|^2 \;<\; \infty.
\]
The computation of $\tilde\Theta$ in the proof of Theorem~\ref{thm:framed-mass}(iv),
carried out on the canonical pullback, then proceeds verbatim, a.e., yielding
$\tilde\Theta = F^*\Theta$; the area formula converts this into the equality of
total masses. This is well defined: if two scalings of the same equation are measurable framed,
their masses coincide. In particular the canonical pullback of
Theorem~\ref{thm:measurable-mass}(iv) is one such representative, so the mass it
computes is the mass of the equation in this sense.
\end{proof}

\begin{proposition}[Well-definedness on scaling classes]\label{prop:admissible}
Call a framed equation \emph{admissible} if some scaling \eqref{eq:row-scaling} of it is measurable framed in the sense of Definition~\ref{def:measurable}, and define its mass to be the mass of such a representative. This is well defined: if two scalings of the same equation are measurable framed, their masses coincide.
\end{proposition}

\begin{proof}
The two representatives differ by the scaling $c = \Phi'/\Phi$. Since $|\Phi|^2 \ge |\Phi|^2 - |\Psi|^2 \ge \varepsilon_K$ a.e., Lemma~\ref{lem:toolbox}(a) gives $1/\Phi \in \Wcal$, hence $c \in \Wcal$; and $|c|^2 = |\Phi'|^2/|\Phi|^2 \ge \varepsilon'_K\,\|\Phi\|_{L^\infty(K)}^{-2} > 0$ a.e.\ on $K$. Theorem~\ref{thm:measurable-mass}(ii) applies, and the masses agree.
\end{proof}

\begin{corollary}[Conversion at measurable regularity]\label{cor:measurable-conversion}
Let \eqref{eq:skeleton} be elliptic with, on every compact $K \subset \Omega$
and for constants $\varepsilon_K > 0$, $k_K < 1$ depending on $K$:
\[
\delta,\; (a_{21}-a_{12}) \in \Wcal, \quad \delta \ge \varepsilon_K \ \text{a.e.\ on } K;
\qquad
\operatorname*{ess\,sup}_K |\mu| \le k_K;
\qquad
a_{13}, a_{14}, a_{23}, a_{24}, f_1, f_2 \in L^2_{\mathrm{loc}}.
\]
Then the conversion $E + \mu\bar E$ of Theorem~\ref{thm:conversion} is measurable
framed. The frame is independent of $\mu$ and lies in $\Wcal$ by
Lemma~\ref{lem:frame-closed-form}: $\operatorname{Re}\Phi = 1+\sqrt\delta \in \Wcal$
since $\delta \in \Wcal$ is bounded below, and
$\operatorname{Im}\Phi = \tfrac12(a_{21}-a_{12}) \in \Wcal$; the frame determinant
is $|\Phi|^2 - |\Psi|^2 = 4\sqrt\delta \ge 4\sqrt{\varepsilon_K}$ a.e.; $\mu$ is
measurable and locally elliptic; and $\fa, \fb, \ff \in L^2_{\mathrm{loc}}$, since
$\mu \in L^\infty$ multiplies the $L^2_{\mathrm{loc}}$ lower-order data
\eqref{eq:conversion-data}. Consequently $\Mcal$ is defined for such systems and
is unchanged by every recombination of unknowns with $\Wcal$ data, every
recombination of equations, and every orientation-preserving quasiconformal
change of variables.

The hypothesis is strictly weaker than $a_{ij} \in \Wcal$: by
Lemma~\ref{lem:frame-closed-form} the frame sees the principal coefficients only
through $\delta$ and $a_{21}-a_{12}$, so the ratios $\alpha = a_{22}/a_{11}$,
$\beta = -(a_{12}+a_{21})/a_{11}$ --- equivalently $\mu$ --- may be merely
measurable while the frame remains in $\Wcal$ (Remark~\ref{rem:gap}).
\end{corollary}

\begin{remark}[The frame can be smoother than $\mu$]\label{rem:gap}
The regularity gap permitted by Definition~\ref{def:measurable} --- frame in
$\Wcal$, $\mu$ merely measurable --- is realized by a converted system, with no
change of variables. Fix any measurable $\mu : \Omega \to \D$ with
$\operatorname*{ess\,sup}_K |\mu| \le k_K < 1$, recover
$\lambda = i(1+\mu)/(1-\mu) = \ell_1 + i\ell_2$ pointwise (so $\ell_2$ is bounded
above and below on $K$), and set, for a smooth constant $\delta_0 > 0$,
\[
a_{11} = \frac{\sqrt{\delta_0}}{\ell_2}, \qquad
a_{22} = \alpha\, a_{11}, \qquad
a_{12} = a_{21} = -\tfrac{\beta}{2}\,a_{11},
\qquad \alpha = |\lambda|^2,\ \ \beta = -2\ell_1.
\]
These are merely measurable elliptic skeleton coefficients with constant
discriminant $\delta = a_{11}^2\,\Delta/4 = \delta_0$ and $a_{21}-a_{12}=0$, so by
Lemma~\ref{lem:frame-closed-form} the converted frame is the constant
\[
\Phi \equiv 1 + \sqrt{\delta_0}, \qquad \Psi \equiv \sqrt{\delta_0} - 1,
\]
while the Beltrami coefficient is the prescribed $\mu$, as rough as one may choose.
The frame being constant, $W_L(\Phi, \Psi) \equiv 0$, so $N = \Phi\,\fb - \Psi\,\fa$
and the mass density of Definition~\ref{def:mass} is well defined with $\mu$ only
measurable. This is the algebraic counterpart of the quasiconformal mechanism of
Theorem~\ref{thm:measurable-mass}(iv): there the roughness is carried by the gauge
weight $\rho$, here by the scale $a_{11} \sim \ell_2^{-1}$, calibrated against
$\sqrt\Delta$ so that $a_{11}\sqrt\Delta = 2\sqrt\delta$ --- the only combination
the frame reads --- stays smooth. Either way the roughness lies in a direction
fixed by the algebra, where $\mu$ cannot see it.
\end{remark}

\begin{corollary}[Measurable uniformization]\label{cor:uniformization}
Let \eqref{eq:framed} be measurable framed with the global bound $\|\mu\|_{L^\infty(\Omega)} < 1$. Then there is an orientation-preserving quasiconformal homeomorphism under which the equation pulls back, after the scaling of Theorem~\ref{thm:measurable-mass}(iv), to a measurable framed equation over $\tilde\mu = 0$ --- a framed \emph{Vekua} equation --- of equal mass.
\end{corollary}

\begin{proof}
Extend $\mu$ by $0$ to $\C$ and let $G : \C \to \C$ be the quasiconformal homeomorphism with $G_{\bar\zeta} = \mu\,G_\zeta$ a.e.\ given by the measurable Riemann mapping theorem \cite{bojarski, aim}; set $F := G^{-1}$, restricted to $G(\Omega)$ and quasiconformal by Lemma~\ref{lem:toolbox}(b). Differentiating $G \circ F = \mathrm{id}$ a.e.\ by the chain rule of Lemma~\ref{lem:toolbox}(b),
\[
0 \;=\; (G \circ F)_{\bar z}
\;=\; (G_\zeta \circ F)\,F_{\bar z} + (G_{\bar\zeta} \circ F)\,\overline{F_z}
\;=\; (G_\zeta \circ F)\,\bigl[\,F_{\bar z} + (\mu \circ F)\,\overline{F_z}\,\bigr],
\]
and $G_\zeta \neq 0$ a.e., so $F_{\bar z} + (\mu \circ F)\,\overline{F_z} = 0$ a.e.\ --- which is precisely the numerator of $\tilde\mu$ in \eqref{eq:diffeo-law}. Hence $\tilde\mu = 0$, and Theorem~\ref{thm:measurable-mass}(iv) gives the equality of masses.
\end{proof}

\begin{remark}\label{rem:smallest-category}
Corollary~\ref{cor:uniformization} is not available within the $C^1$ category of Sections~\ref{sec:substitution}--\ref{sec:mass}: the uniformizing map of a non-smooth --- or even merely continuous --- $\mu$ is in general only $W^{1,2}_{\mathrm{loc}}$, so the smooth theory cannot transform its own Beltrami coefficient away without leaving its category. The measurable class is the natural extension closed under its own uniformization, and within it every mass is the mass of a framed Vekua equation. This also locates the construction relative to the pipeline of \cite{mass}: the pipeline differentiates the moving generator of the variable elliptic structure \cite{ves}, and its intermediate objects require classical derivatives of the principal coefficients; here the generator is never instantiated, the only weak derivative taken falls on the frame, and the invariant accordingly survives in the class --- Bojarski's \cite{bojarski} --- where the Beltrami coefficient is merely measurable.
\end{remark}

\begin{remark}[Comparison with Vekua's reduction]\label{rem:vekua}
The single derivative of the present theory falls on the same object as the
single derivative of the classical reduction of \cite[\S7]{vekua}. There the
skeleton \eqref{eq:skeleton} is carried to the Cauchy--Riemann canonical form
by solving a Beltrami equation for a change of coordinates and then applying
the unknown-substitution $U = \sqrt{\delta}\,u$, $V = v - \tfrac12(a_{12}-a_{21})u$,
whose differentiation produces the zeroth-order coefficients of the canonical
form,
\[
a_* = a_1' - \partial_\xi\sqrt{\delta} - \tfrac12\,\partial_\eta(a_{12}-a_{21}),
\qquad
c_* = a_2' - \partial_\eta\sqrt{\delta} + \tfrac12\,\partial_\xi(a_{12}-a_{21}),
\]
in which $\delta$ is Vekua's discriminant $\Delta$ of \cite[(7.6)]{vekua}. The two
functions differentiated, $\sqrt{\delta}$ and $a_{12}-a_{21}$, are by
Lemma~\ref{lem:frame-closed-form} the real and imaginary parts of the frame,
$\operatorname{Re}\Phi = 1+\sqrt{\delta}$, $\operatorname{Im}\Phi = -\tfrac12(a_{12}-a_{21})$;
the substitution is, in the present language, a straightening
(Proposition~\ref{prop:straighten}), recombining the unknown by the frame data;
and the derivative terms in $a_*, c_*$ are precisely the substitution-contamination
\eqref{eq:contamination}. Forming no invariant, \cite{vekua} leaves this
contamination in the reduced equation as its lower-order coefficients, whereas
the numerator field $N = \Phi\,\fb - \Psi\,\fa - W_L(\Phi,\Psi)$ is the combination
in which the same term cancels against the $L$-Wronskian
(Remark~\ref{rem:consistency}). The two reductions spend one derivative on one
object --- the frame; one displays it as a coefficient, the other absorbs it
into the invariant.

The regularity costs then separate. The reduction of \cite[\S7]{vekua} requires
the principal coefficients in $D_{m+1,p}$ and the lower-order data in $D_{m,p}$
with $p > 2$, and its coordinate change is the solution $\zeta \in D_{m+2,p}$ of
the Beltrami equation, available only for $q \in D_{m+1,p}$: the conformal datum
$q$ must be once differentiable and a partial differential equation must be
solved. The present reduction differentiates nothing in the conversion
(Theorem~\ref{thm:conversion}) and solves no equation; the lone derivative, that
of $W_L$, asks only for the frame in $\Wcal$ --- equivalently $\sqrt{\delta}$ and
$a_{12}-a_{21}$ in $W^{1,2}_{\mathrm{loc}}$ --- while the Beltrami coefficient $\mu$,
which carries the conformal structure that \cite{vekua} must integrate away,
remains merely measurable (Corollary~\ref{cor:measurable-conversion},
Remark~\ref{rem:gap}). The classical reduction's extra regularity does not purchase the normalization
of the principal part to $\partial_{\bar\zeta}$ itself --- that is available far
more cheaply, by straightening before uniformizing (Remark~\ref{rem:reorder}) ---
but only its classical, $C^1$-chart form; and the normalization, however
performed, is the step under which the mass becomes invisible.
\end{remark}

\begin{remark}[Comparison with Bojarski's reduction]\label{rem:bojarski}
The conversion of Section~\ref{sec:conversion} and Bojarski's complex reduction
\cite[\S2.1]{bojarski} begin from one and the same object and diverge only in the
combination of $E$ and $\bar E$ they form. Bojarski's complex equation
\cite[(2.4)]{bojarski} is the slot form \eqref{eq:slot-form}: his coefficients
$\lambda,\mu$ --- not those of the present paper --- are, by \eqref{eq:rs},
$1+\lambda = s$ and $\mu = r$, so the four-slot residual is shared verbatim. From
it, \cite[(2.6)]{bojarski} eliminates $\overline{w_{\bar z}}$ between $E$ and
$\bar E$ and normalizes the $w_{\bar z}$-coefficient to one,
\[
w_{\bar z} - q_1\,w_z - q_2\,\overline{w_z} = A\,w + B\,\bar w + C,
\qquad |q_1| + |q_2| \le q_0 < 1,
\]
carrying the conjugate-linearity in the \emph{derivative} term $q_2\,\overline{w_z}$.
The present conversion instead forms $E + \mu\,\bar E$ \eqref{eq:conversion-data},
retaining all four slots as $\Phi\,Lw + \Psi\,\overline{Mw}$; in particular it keeps
$\overline{w_{\bar z}}$ --- as the term $-\Psi\mu\,\overline{w_{\bar z}}$ inside
$\overline{Mw}$ --- precisely the slot Bojarski removes, while both forms retain
$\overline{w_z}$ (his $q_2$, our leading coefficient of $\overline{Mw}$). Each
reduction is pure pointwise algebra valid at measurable coefficients: Bojarski's
elimination is licensed by the ellipticity bound that holds
$\big||\mu|^2 - |1+\lambda|^2\big|$ away from zero, ours by
$|\Phi|^2 - |\Psi|^2 = 2a_{11}\sqrt\Delta$ \eqref{eq:frame-det}, and neither
differentiates a coefficient.

The two forms part company over the second principal coefficient. Bojarski's
$q_2$ is not removable at the level of the equation: his passage to a single
Beltrami coefficient, \cite[(4.8)]{bojarski},
$q_0 = q_1 + q_2\,\overline{w_z}/w_z$, depends on the solution. The
$\overline{w_z}$ coupling does not admit the recombination
$w \mapsto \varphi w + \psi\bar w$ --- feeding $\bar w'$ through it produces
derivatives of $(\varphi,\psi)$ with no principal slot to receive them. The frame
is that slot: the recombination closes (Proposition~\ref{prop:substitution}), and
the straightening (Proposition~\ref{prop:straighten}) removes $\Psi$ --- Bojarski's
$q_2$ --- at the level of the equation, its sole cost the term
$L\nu = W_L(\Phi,\Psi)/\Phi^2$ it deposits in $\Bcal_{\mathrm{str}}$
\eqref{eq:straightened}: that is, the numerator field $N$, that is, the mass. What
Bojarski can only linearize per solution, the frame removes once and for all, at
the price of the invariant; this is the algebraic content of the claim that the
conjugate Vekua coupling admits the recombination and Bojarski's does not.
Bojarski himself forms no invariant but exponentiates: in his representation
$w = f(\chi)\,e^{\varphi}$ \cite[Thm.~4.1]{bojarski} the conjugate coefficient $B$
enters through $h = A + B\,\bar w/w$ and is absorbed into $\varphi$, while the
residue of that same $B$ that no gauge exponentiates away is exactly
$|b_{(F,G)}|^2$, the mass density (Proposition~\ref{prop:bers}). His representation
makes the coupling disappear; the present invariant measures it --- as with
\cite[\S7]{vekua} (Remark~\ref{rem:vekua}), one disposal and one measurement of a
single object.
\end{remark}

\begin{remark}[Straightening before uniformizing]\label{rem:reorder}
The present theory reaches Vekua's target as well --- the standard Vekua equation
$w_{\bar\zeta} + \Acal\,w + \Bcal\,\bar w = \Fcal$ --- and at the regularity of
Corollary~\ref{cor:measurable-conversion} rather than that of \cite[\S7]{vekua}.
Both reductions perform the same two operations, one differentiation of the gauge
and one uniformization of $\mu$ (a Beltrami solve), and differ only in their order.
\cite{vekua} uniformizes first and differentiates the gauge afterward, in the new
coordinates: his substitution (7.18) and its derivative (7.20) are classical only
if the uniformizing chart is $C^1$, i.e.\ $\zeta \in D_{m+2,p}$, which forces
$q \in D_{m+1,p}$ and hence $\mu$ once differentiable. The present route
differentiates the gauge first --- the straightening of
Proposition~\ref{prop:straighten}, carried out in the original chart, where it
asks only for the frame in $\Wcal$ and uses $\mu$ as a bounded multiplier --- and
uniformizes afterward: the measurable Beltrami solve of
Corollary~\ref{cor:uniformization} is a quasiconformal change of variables under
which $\mu \mapsto 0$ and, by the diffeomorphism law
(Proposition~\ref{prop:diffeo}, Theorem~\ref{thm:measurable-mass}(iv)), the
lower-order data merely compose, no derivative being taken. The composite carries
any admissible framed equation, through its straightened Beltrami--Vekua slice, to
a standard Vekua equation, with the single gauge derivative of the straightening
as the only differentiation anywhere and $\mu$ never differentiated.

The regularity cost of the composite thus depends on the order: differentiating
the gauge before uniformizing decouples it from the smoothness of the uniformizing
chart and leaves $\mu$ merely measurable, whereas differentiating it afterward
couples the two and inflates the requirement to the full principal part of
\cite[\S7]{vekua}. What the order cannot change is the smoothness of the chart
itself --- with $\mu$ only measurable the uniformizing map is only quasiconformal,
so the Vekua equation reached has $L^2_{\mathrm{loc}}$ coefficients and holds
almost everywhere, the generalized equation of Bojarski's class \cite{bojarski}
rather than the $D_{m,p}$ presentation of \cite[\S7]{vekua}. This is no loss of
standardness: the $m=0$ reduction of \cite{vekua} already yields only $L^p$
coefficients, and the sole thing its extra regularity on $\mu$ secures is the
$C^1$ chart.
\end{remark}

\section{Coda: the pure frame as a Bers generating pair}\label{sec:coda}

The frame entered this paper as the one piece of data that carries a derivative,
inserted into the equation so that the recombination $w \mapsto \varphi w + \psi \bar w$
could close the class and so that the $L$-Wronskian could absorb the substitution's
defect into the invariant. Strip the equation of everything else --- set
$\mu = \fa = \fb = \ff = 0$ in \eqref{eq:framed} --- and what remains is the frame
alone:
\begin{equation}\label{eq:pure-frame}
\Phi\, w_{\bar z} + \Psi\, \overline{w_z} = 0, \qquad
\text{equivalently}\qquad w_{\bar z} = -\nu\,\overline{w_z}, \quad \nu := \Psi/\Phi,\ |\nu| < 1.
\end{equation}
One then sees what the frame has been all along. The pair $(\Phi, \Psi)$ is a
generating pair in the sense of Bers \cite{bers}, the equation
\eqref{eq:pure-frame} is the equation that pair generates, and the pseudo-analytic
mass is the $L^2$ size of Bers' characteristic coefficient --- the term to which
the construction owes its name.

\begin{proposition}[The pure frame is a Bers generating pair]\label{prop:bers}
Let \eqref{eq:framed} have $\mu = \fa = \fb = \ff = 0$ and $C^1$ (or $\Wcal$) frame,
$|\Phi| > |\Psi|$, so that it reads \eqref{eq:pure-frame}. Set
\[
F := \Phi + \Psi, \qquad G := i\,(\Phi - \Psi),
\qquad\text{equivalently}\qquad
\Phi = \tfrac12(F - iG), \quad \Psi = \tfrac12(F + iG).
\]
Then:
\begin{itemize}
\item[(i)] $(F, G)$ is a generating pair, with Bers' nondegeneracy equal to the
frame determinant:
\[
\operatorname{Im}(\bar F G) = |\Phi|^2 - |\Psi|^2 > 0 .
\]
\item[(ii)] Decomposing $w = \phi F + \psi G$ with $\phi, \psi$ real and writing the
second-kind coordinate $\omega := \phi + i\psi$, the first-kind function is the frame
recombination
\[
w \;=\; \phi F + \psi G \;=\; \Phi\,\omega + \Psi\,\bar\omega,
\]
and the pure-frame equation \eqref{eq:pure-frame}, read in $\omega$, is exactly Bers'
pseudoanalyticity condition:
\[
\phi_{\bar z} F + \psi_{\bar z} G \;=\; \Phi\,\omega_{\bar z} + \Psi\,\overline{\omega_z} \;=\; 0 .
\]
Its solutions are the $(F, G)$-pseudoanalytic functions of the second kind; the
pseudoanalytic constants $\omega \equiv 1, i$ are the generators $F, G$.
\item[(iii)] The first-kind function satisfies the Vekua equation
$w_{\bar z} = a_{(F,G)}\,w + b_{(F,G)}\,\bar w$ with characteristic coefficients
\[
a_{(F,G)} = \frac{\bar\Phi\,\Phi_{\bar z} - \bar\Psi\,\Psi_{\bar z}}{|\Phi|^2 - |\Psi|^2},
\qquad
b_{(F,G)} = -\,\frac{N}{|\Phi|^2 - |\Psi|^2} = \frac{\Phi^2\,\nu_{\bar z}}{|\Phi|^2 - |\Psi|^2},
\]
where $N = -W_L(\Phi, \Psi) = -\Phi^2\,\nu_{\bar z}$ is the numerator field of
Definition~\ref{def:mass}. Consequently the mass density \eqref{eq:mass-def} is the
squared modulus of Bers' coefficient,
\[
\Theta \;=\; \frac{|N|^2}{(|\Phi|^2 - |\Psi|^2)^2}\,dx\,dy
\;=\; \bigl| b_{(F,G)} \bigr|^2\, dx\, dy
\;=\; \frac{|\nu_{\bar z}|^2}{(1 - |\nu|^2)^2}\, dx\, dy,
\qquad \Mcal = \int_\Omega \bigl|b_{(F,G)}\bigr|^2 .
\]
\end{itemize}
\end{proposition}

\begin{proof}
(i) $\bar F G = i\,(\bar\Phi + \bar\Psi)(\Phi - \Psi)
= i\bigl[(|\Phi|^2 - |\Psi|^2) + (\bar\Psi\Phi - \bar\Phi\Psi)\bigr]$, and
$\bar\Psi\Phi - \bar\Phi\Psi$ is purely imaginary, so
$\operatorname{Im}(\bar F G) = |\Phi|^2 - |\Psi|^2$, positive by the frame condition.

(ii) With $\phi = \tfrac12(\omega + \bar\omega)$ and $\psi = \tfrac1{2i}(\omega - \bar\omega)$,
\[
\phi F + \psi G = \omega\,\tfrac12(F - iG) + \bar\omega\,\tfrac12(F + iG) = \Phi\,\omega + \Psi\,\bar\omega .
\]
Since $\phi, \psi$ are real, $\overline{\phi_z} = \phi_{\bar z}$, so
$\phi_{\bar z} = \tfrac12(\omega_{\bar z} + \overline{\omega_z})$ and
$\psi_{\bar z} = -\tfrac{i}{2}(\omega_{\bar z} - \overline{\omega_z})$; collecting,
\[
\phi_{\bar z} F + \psi_{\bar z} G
= \tfrac12(F - iG)\,\omega_{\bar z} + \tfrac12(F + iG)\,\overline{\omega_z}
= \Phi\,\omega_{\bar z} + \Psi\,\overline{\omega_z},
\]
which is \eqref{eq:pure-frame}. The constants $\omega \equiv 1, i$ give
$\Phi + \Psi = F$ and $i(\Phi - \Psi) = G$.

(iii) The generators are pseudoanalytic constants,
$F_{\bar z} = a F + b \bar F$ and $G_{\bar z} = a G + b \bar G$; solving this $2\times2$
system, whose determinant is $F\bar G - \bar F G = -2i\,(|\Phi|^2 - |\Psi|^2)$ by (i),
gives the displayed $a_{(F,G)}, b_{(F,G)}$ after expansion, the $b$-numerator reducing
to $-2i\,(\Phi\Psi_{\bar z} - \Psi\Phi_{\bar z}) = -2i\,\Phi^2\nu_{\bar z}$. By
\eqref{eq:quotient} at $\mu = 0$, $\Phi^2\nu_{\bar z} = W_L(\Phi, \Psi)$, and since
$\fa = \fb = 0$ the numerator field is $N = -W_L(\Phi, \Psi) = -\Phi^2\nu_{\bar z}$;
hence $b_{(F,G)} = -N/(|\Phi|^2 - |\Psi|^2)$ and
$|b_{(F,G)}|^2 = |N|^2/(|\Phi|^2 - |\Psi|^2)^2$, which is the density of
\eqref{eq:mass-def} at $\mu = 0$.
\end{proof}

The dictionary closes the circle the paper opened. The recombination
$w \mapsto \varphi w + \psi\bar w$, which the frame was built to absorb and which the
conjugate Vekua coupling alone admits, is at the pure-frame extreme exactly Bers'
passage between functions of the first and second kind: the frame is the covariant
name for the generating pair, its determinant is Bers' nondegeneracy, its trivial
slice $(\Phi, \Psi) = (1, 0)$ is the pair $(1, i)$ and the ordinary
Cauchy--Riemann equation, and the $L$-Wronskian that Section~\ref{sec:mass} placed
inside $N$ so that the substitution's defect would cancel is, here, Bers'
characteristic coefficient $b_{(F,G)}$ up to the frame determinant. The mass is the
$L^2$ norm of that coefficient. Its companion $a_{(F,G)}$ is pure gauge, removed by
the $\C$-linear action of \cite{mass}; only $|b_{(F,G)}|^2$ is invariant, and it
measures the one part of the pair's failure to be holomorphic that no gauge can
erase. The mass vanishes precisely when $\nu = \Psi/\Phi$ is holomorphic --- the flat
case, in which the generating pair reduces, after a holomorphic recombination, to the
analytic one. This is why the invariant deserves the adjective it has carried since
\cite{mass}: it is the pseudo-analytic mass in Bers' own sense, the obstruction to a
generating pair being analytic.

Finally, the correspondence is indifferent to regularity in exactly the way
Section~\ref{sec:measurable} is. A frame in $\Wcal$ is a generating pair with
$W^{1,2}_{\mathrm{loc}} \cap L^\infty_{\mathrm{loc}}$ generators, its
nondegeneracy $|\Phi|^2 - |\Psi|^2$ bounded below on compacts, and its coefficient
$b_{(F,G)} \in L^2_{\mathrm{loc}}$; the measurable framed theory
(Theorem~\ref{thm:measurable-mass}) is Bers' theory of pseudoanalytic functions
carried to Sobolev generating pairs, and the uniformization of
Corollary~\ref{cor:uniformization} is the statement that every such pair is, up to a
quasiconformal change of variables of equal mass, a pair over the standard conformal
structure --- a generating pair in the plane.

\subsection*{Use of Generative AI Tools}
\medskip

The author discloses the use of Anthropic's Claude (Claude Fable 5, accessed
through the Claude.ai web interface in June 2026) in the preparation of this
manuscript. The tool was used as follows:

\begin{enumerate}
\item[(i)] \emph{Exploratory dialogue.} The algorithmic organization of the
construction --- in particular the closed conversion $E + \mu\bar E$, the
exact substitution law \eqref{eq:N-law} of the numerator field, the
$C^1$ diffeomorphism law \eqref{eq:diffeo-law}, and the measurable-regularity
formulation of Section~\ref{sec:measurable} --- emerged in iterative research sessions, building on the companion papers \cite{mass, abs}.

\item[(ii)] \emph{Symbolic verification and stress-testing.} The conversion identities, the
law \eqref{eq:N-law}, and the diffeomorphism identities of
Section~\ref{sec:diffeo} were stress-tested by exact rational computer algebra, selected cases verifications, with scripts produced with the tool's assistance and re-run, and further verified and analyzed by the author.

\item[(iii)] \emph{Drafting and revision of prose.} The manuscript was
drafted in iterative dialogue; all claims and their precise wording were
reviewed by the author.
\end{enumerate}

The author takes full responsibility for
the correctness, accuracy, originality, and integrity of all content.

\subsection*{Disclosure of interest}

The author reports there are no competing interests to declare.

\end{document}